\documentclass[reqno, 11pt]{amsart}
\usepackage{amsmath}

\begin{document}

\newtheorem{lem}{Lemma}[section]
\newtheorem{pro}[lem]{Proposition}
\newtheorem{thm}[lem]{Theorem}
\newtheorem{cor}[lem]{Corollary}


\title{Morse homology of manifolds with boundary revisited}

\author{Manabu Akaho}
\address{Department of Mathematics and Information Sciences,
Tokyo Metropolitan University, 
1-1 Minami-Ohsawa, Hachioji, Tokyo 192-0397, Japan}
\email{akaho@tmu.ac.jp}

\thanks{Supported by JSPS Grant-in-Aid for Young Scientists (B)}

\maketitle

\begin{abstract}
This re-certifying paper describes the details of
the Morse homology of manifolds with boundary, introduced in \cite{ak1}, 
in terms of handlebody decompositions.
First we carefully observe Riemannian metrics and Morse functions on manifolds with boundary
so that their gradient vector fields are tangent to the boundary;
secondly we confirm the stable manifolds and the unstable manifolds 
of critical points, and rigorously construct handlebody decompositions;
and finally we re-certify that 
our Morse homology of manifolds with boundary is isomorphic to the absolute singular homology
through connecting homomorphisms.
\end{abstract}

\section{\bf Introduction}

Inspired by H. Hofer \cite{ho} and symplectic field theory \cite{egh},
the author introduced a variant of Floer homology in \cite{ak1}
for Lagrangian submanifolds 
with Legendrian cylindrical end
in a symplectic manifold with concave end,
which can be thought as an infinite dimensional version of the following 
Morse homology of manifolds with boundary.

Let $M$ be an $n$-dimensional oriented compact manifold with boundary $N$.
Denote by $N_1, N_2, \ldots,N_m$ the connected components of $N$.
We fix a collar neighborhood $N\times [0,1)\subset M$, 
and denote by $r$ the standard coordinate on the $[0,1)$-factor.
Then we consider a Riemannain metric $g$
on $M\setminus N$ such that, for $i=1,\ldots,m$,
\[
g|_{N_i\times (0,1)}(x,r)=r^2g_{N_i}(x)+dr\otimes dr,
\]
where $g_{N_i}$ is a Riemannian metric on $N_i$,
and we consider a Morse--Smale function $f$ on $M\setminus N$ 
such that,  for $i=1,\ldots, m$,
\[
f|_{N_i\times(0,1)}(x,r)=r^2f_{N_i}(x)+c_i,
\]
where $f_{N_i}:N_i\to\mathbb{R}$ is a Morse--Smale function on $N_i$
and $c_i\in\mathbb{R}$ is a constant.

Let $Cr_k(f)$ be the set of the critical points $p\in M\setminus N$ of $f$ with Morse index $k$,
and $Cr_k^+(f_{N_i})$ the set of the critical points $\gamma\in N_i$ of $f_{N_i}$ 
with Morse index $k$ and $f_{N_i}(\gamma)>0$,
and similarly let
$Cr_k^-(f_{N_i})$ be the set of the critical points $\delta\in N_i$ of $f_{N_i}$
with Morse index $k$ and $f_{N_i}(\delta)<0$.
We put  $Cr_k^+(f_N):=\bigcup_{i=1}^mCr_k^+(f_{N_i})$ and 
$Cr_k^-(f_N):=\bigcup_{i=1}^mCr_k^-(f_{N_i})$.
Then we define our Morse complex.
Let $CM_k(f)$ be a free $\mathbb{Z}$ module 
\[
CM_k(f):=\bigoplus_{p\in Cr_k(f)}\mathbb{Z}p
\oplus
\bigoplus_{\gamma\in Cr_k^+(f_N)}\mathbb{Z}\gamma,
\]
and $\partial_k:CM_k(f)\to CM_{k-1}(f)$ a linear map, 
for $p\in Cr_k(f)$,
\begin{eqnarray*}
\partial_k p&:=&
\sum_{p'\in Cr_{k-1}(f)}\sharp\mathcal{M}(p,p')p'
+\sum_{\gamma\in Cr_{k-1}^+(f_N)}\sharp\mathcal{M}(p,\gamma)\gamma,
\end{eqnarray*}
and for $\gamma\in Cr_k^+(f_N)$,
\begin{eqnarray*}
\partial_k\gamma &:=&
\sum_{p\in Cr_{k-1}(f)}\sum_{\delta\in Cr_{k-1}^-(f_N)}
\sharp\mathcal{N}(\gamma,\delta)\sharp\mathcal{M}(\delta,p)p
\\
&&
+\sum_{\gamma'\in Cr_{k-1}^+(f_N)}\sum_{\delta\in Cr_{k-1}^-(f_N)}
\sharp\mathcal{N}(\gamma,\delta)\sharp\mathcal{M}(\delta,\gamma')\gamma'
\\
&&
+\sum_{\gamma'\in Cr_{k-1}^+(f_N)}\sharp\mathcal{N}(\gamma,\gamma')\gamma'.
\end{eqnarray*}
We give the precise definition of $\partial_k$ in Section 7.
An important remark is that $\delta\in Cr_k^-(f_N)$ is not a generator of $CM_k(f)$.
Then our main theorem is that:

\begin{thm}
$(CM_*(f),\partial_*)$ is a chain complex, i.e. $\partial_{k-1}\circ\partial_k=0$, 
and the homology is isomorphic to the {\it absolute} singular homology of~$M$.
\end{thm}

As a corollary we obtain the following Morse type inequalities.

\begin{cor}
\[
\sharp Cr_k(f)+\sharp Cr_k^+(f_N)\geq \dim H_k(M;\mathbb{R}).
\]
\end{cor}

There are several remarks on other related Morse homology.
In \cite{km}, motivated by Seiberg--Witten Floer homology,
Kronheimer--Mrowka also observed Morse homology of manifolds with boundary;
they considered the double of a manifold with boundary and involution invariant Morse functions.
In \cite{la} F. Laudenbach also studied Morse homology of manifolds with boundary;
his gradient vector field are also tangent to the boundary and 
his Morse complex counts trajectories of pseudo-gradient vector fields.

This paper consists of the following sections:
first, in Section 2 we carefully observe 
Riemannian metrics and Morse functions on manifolds with boundary;
in Section 3 we confirm the stable manifolds and the unstable manifolds of critical points,
and fix their orientations;
then, in Section 4 we rigorously construct handlebody decompositions,
and in Section 5 we introduce relative cycles of critical points, 
which is a new technique and very important for the future applications;
moreover, in Section 6 we prepare moduli spaces of gradient trajectories; 
and finally, in Section 7 we recall our Morse complex of manifolds with boundary,
introduced in \cite{ak1}, 
and re-certify that our Morse homology is isomorphic to the absolute singular homology
through connecting homomorphisms.
Although it is important for Floer theory
to consider compactifications of the moduli spaces of gradient trajectories,
we do not mention them in this paper;
the reader may refer to~\cite{ak1}.

\section{\bf Riemannian metrics and Morse functions}

In this section, we carefully observe Riemannian metrics and Morse functions 
on manifolds with boundary
so that their gradient vector fields are tangent to the boundary.

Let $M$ be an $n$-dimensional oriented compact manifold with boundary $N$.
We denote by $N_1, N_2, \ldots, N_m$ the connected components of $N$.
Fix a collar neighborhood $N\times [0,1)\subset M$,
and denote by $r$ the standard coordinate on the $[0,1)$-factor.

Let $g$ be a Riemannian metric on $M\setminus N$, 
and $f$ a smooth function on $M\setminus N$.
Just for simplicity, we consider $g$ whose restriction on the collar neighborhood is
\[
g|_{N_i\times(0,1)}=ag_{N_i}+dr\otimes dr,
\]
where $a:(0,1)\to \mathbb{R}$ is a smooth function and $g_{N_i}$ is a Riemannian metric on $N_i$.
On the other hand, for the gluing analysis of gradient trajectories in Morse homology,
we require the gradient vector field $X_f$ of $f$ with respect to $g$ to be the following form
on the collar neighborhood
under the coordinate change of $r\in(0,1)$ and $t\in(-\infty,0)$ by $r=e^t$:
\[
X_f|_{N_i\times(-\infty,0)}=X_{f_{N_i}}+h_{N_i}\frac{\partial}{\partial t},
\]
where $f_{N_i}$ and $h_{N_i}$ are smooth functions on $N_i$, and 
$X_{f_{N_i}}$ is the gradient vector field of $f_{N_i}$ with respect to $g_{N_i}$ on $N_i$.
Note that
\[
X_{f_{N_i}}+h_{N_i}\frac{\partial}{\partial t}=X_{f_{N_i}}+h_{N_i}r\frac{\partial}{\partial r}.
\]

\begin{lem} 
Suppose that $N$ is connected.
Let $g$ be a Riemannian metric on $N\times(0,1)$ such that
\[
g=ag_N+dr\otimes dr, 
\]
where $a:(0,1)\to\mathbb{R}$ is a smooth function and 
$g_N$ is a Riemannain metric on $N$, 
and let $f:N\times(0,1)\to \mathbb{R}$ be a smooth function whose gradient vector field $X_f$ 
with respect to $g$ is 
 \[
X_f=X_{f_N}+h_Nr\frac{\partial}{\partial r},
\]
where $f_N$ and $h_N$ are non-constant smooth functions on $N$,
and $X_{f_N}$ is the gradient vector field of $f_N$ with respect to $g_N$ on $N$.
Then 
\begin{eqnarray*}
g&=&(Ar^2+B)g_N+dr\otimes dr,
\\
f&=&(Ar^2+B)f_N+Cr^2+D,
\end{eqnarray*}
where $A\neq0, B,C,D\in \mathbb{R}$ are constants.
\end{lem}

\begin{proof}
We write $f=f(x,r)$, $f_N=f_N(x)$, $h_N=h_N(x)$ and $a=a(r)$ for $(x,r)\in N\times (0,1)$,
and we denote by $d_N$ the exterior derivative on $N$. 
From $g$ and $X_f$, 
\[
g(X_f,\cdot)=a(r)d_Nf_N(x)+h_N(x)rdr
\]
On the other hand,
\[
df(x,r)=d_Nf(x,r)+\frac{\partial f(x,r)}{\partial r}dr.
\]
So we have 
\begin{eqnarray}
d_Nf(x,r)&=&d_N\{a(r)f_N(x)\},
\\
\frac{\partial f(x,r)}{\partial r}&=& h_N(x)r.
\end{eqnarray}
Since $N$ is connected, and from (1),
\[
f(x,r)=a(r)f_N(x)+c(r),
\]
where $c=c(r)$ is a smooth function on $(0,1)$,
and then 
\begin{equation}
\frac{\partial f(x,r)}{\partial r}=\frac{da(r)}{dr}f_N(x)+\frac{dc(r)}{dr}.
\end{equation}
From (2) and (3),
\[
h_N(x)=\frac{1}{r}\frac{da(r)}{dr}f_N(x)+\frac{1}{r}\frac{dc(r)}{dr}.
\]
Since we assume that $f_N(x)$ and $h_N(x)$ are non-constant smooth functions, 
\begin{eqnarray*}
\frac{1}{r}\frac{da(r)}{dr}&=&2A,
\\
\frac{1}{r}\frac{dc(r)}{dr}&=&2C,
\end{eqnarray*}
where $A\neq0$ and  $C\in\mathbb{R}$ are constants, and hence
\begin{eqnarray*}
a(r)&=& Ar^2+B,
\\
c(r)&=& Cr^2+D,
\end{eqnarray*}
where $B,D\in\mathbb{R}$ are constants.
Then we obtain 
\begin{eqnarray*}
g(x,r)&=&(Ar^2+B)g_N(x)+dr\otimes dr,
\\
f(x,r)&=&(Ar^2+B)f_N(x)+Cr^2+D.
\end{eqnarray*}
\end{proof}

\begin{cor}
\[
X_f=X_{f_N}+2(Af_N+C)r\frac{\partial}{\partial r}.
\]
\end{cor}

We call a Riemannian metric $g$ on $M\setminus N$
{\it cone end} if $g$ satisfies
\[
g|_{N_i\times (0,1)}=r^2g_{N_i}+dr\otimes dr,
\]
where $g_{N_i}$ is a Riemannian metric on $N_i$,
and we call a Morse function $f$ on $M\setminus N$ 
{\it cone end} if $f$ satisfies
\[
f|_{N_i\times(0,1)}=r^2f_{N_i}+c_i,
\]
where $f_{N_i}$ is a Morse function on $N_i$ and $c_i\in \mathbb{R}$ is a constant.
On the other hand,
in \cite{ak1} we used Riemannian metrics $g$ and Morse functions $f$ on $M\setminus N$
such that
\begin{eqnarray*}
g|_{N_i\times(0,1)}&=&rg_{N_i}+r^{-1}dr\otimes dr,
\\
f|_{N_i\times(0,1)}&=&rf_{N_i}+c_i,
\end{eqnarray*}
which we called {\it horn end}. 
The following lemma  implies that there is no essential difference 
between cone end and horn end for our purpose;
the case of $\bar{a}=0$ is cone end, and $a=-1$ is horn end.
\begin{lem}
For $a+2\neq0$ and $\bar{a}+2\neq0$, 
let $\bar{r}^{\bar{a}+2}=\left(\frac{\bar{a}+2}{a+2}\right)^2r^{a+2}$,
$\bar{g}_N=\left(\frac{a+2}{\bar{a}+2}\right)^2g_N$ and 
$\bar{f}_N=\left(\frac{a+2}{\bar{a}+2}\right)^2f_N$.
Then 
\begin{eqnarray*}
\bar{r}^{\bar{a}+2}\bar{g}_N
+\bar{r}^{\bar{a}}d\bar{r}\otimes d\bar{r}
&=&
r^{a+2}g_N
+r^adr\otimes dr,
\\
\bar{r}^{\bar{a}+2}\bar{f}_N+c&=&r^{a+2}f_N+c.
\end{eqnarray*}
\end{lem}

\begin{proof}
Direct computations.
\end{proof}

Moreover, instead of cone end or horn end, 
we can also use Riemannian metrics $g$ and Morse functions $f$ on $M\setminus N$, or $M$,
which satisfy
\begin{eqnarray*}
g|_{N_i\times(0,1)}&=&(r^2+1)g_{N_i}+dr\otimes dr,
\\
f|_{N_i\times(0,1)}&=&(r^2+1)f_{N_i}+c_i
\end{eqnarray*}
since their gradient vector field on $N_i\times(0,1)$ is 
completely the same as the one of cone end:
\[
X_f|_{N_i\times(0,1)}=X_{f_{N_i}}+2f_{N_i}r\frac{\partial}{\partial r}.
\]
We call such Riemannian metrics and Morse functions {\it doubling end}.

We remark that, for all the types of pairs of a Riemannian metric and a Morse function above,
we can define their Morse complex in the same way.

In this paper we use cone end.

\begin{lem}
Let $f$ be a cone end Morse function. 
If $\gamma\in N_i$ is a critical point of $f_{N_i}$,
then $f_{N_i}(\gamma)\neq 0$.
\end{lem}

\begin{proof}
Since $df|_{N_i\times (0,1)}=r^2df_{N_i}+2rf_{N_i}dr$ and 
the critical points of a Morse function are isolated,
$f_{N_i}(\gamma)\neq 0$. 
\end{proof}

Hence we can divide the critical points $x$ of $f_{N_i}$ into two groups;
one is $f_{N_i}(x)>0$, and the other is $f_{N_i}(x)<0$.

Moreover we have the following lemma:

\begin{lem}
Let $g$ and $f$ be cone end.
Then there is no map $u:\mathbb{R}\to M\setminus N$
which satisfies $du/dt=-X_f\circ u$ with 
$\lim_{t\to-\infty}u(t)\in N_i\times\{0\}$ and $\lim_{t\to\infty}u(t)\in N_i\times\{0\}$.
\end{lem}

\begin{proof}
Suppose that a map $u:\mathbb{R}\to M\setminus N$ satisfies
$du/dt=-X_f\circ u$ with $\lim_{t\to-\infty}u(t)\in N_i\times\{0\}$ and 
$\lim_{t\to\infty}u(t)\in N_i\times\{0\}$. 
Since $du/dt=-X_f\circ u$, $\lim_{t\to-\infty}f(u(t))>\lim_{t\to\infty}f(u(t))$,
which contradicts $\lim_{t\to-\infty}f(u(t))=\lim_{t\to\infty}f(u(t))=c_i$.
\end{proof}

On the other hand,
there may exist 
\begin{itemize}
\item a non-constant map $u:\mathbb{R}\to N_i$ which satisfies $du/dt=-X_{f_{N_i}}\circ u$,
and 
\item a map $u:\mathbb{R}:\to M\setminus N$
which satisfies
$du/dt=-X_f\circ u$ with
$\lim_{t\to-\infty}u(t)\in N_i\times\{0\}$ and $\lim_{t\to\infty}u(t)\in N_j\times\{0\}$
if $c_i>c_j$;
this was pointed out by T. Nishino and Y. Nohara.
\end{itemize}

We remark that, although their proofs need slight modifications, Lemma 2.4 and Lemma 2.5 
also hold for horn end and doubling end. 

\section{\bf Stable manifolds and unstable manifolds}

First we prepare notation.
Let $M$ be an $n$-dimensional oriented compact manifold with boundary $N$ as before,
and $g$ and $f$ a cone end Riemannian metric and a cone end Morse function 
on $M\setminus N$, respectively.
We fix an orientation of $N$ so that, for  
an oriented basis $\{v_1,\ldots,v_{n-1}\}$ of $T_pN$
and an outward-pointing vector $v_{out}\in T_pM$, 
the orientations of $\{v_{out}, v_1,\ldots,v_{n-1}\}$ and $T_pM$ coincide.

We define $Cr_k(f)$ to be the set of 
the critical points $p\in M\setminus N$ of $f$ with Morse index $k$,
and $Cr_k^+(f_{N_i})$ the set of the critical points $\gamma\in N_i$ of $f_{N_i}$ 
with Morse index $k$ and $f_{N_i}(\gamma)>0$,
and similarly we define
$Cr_k^-(f_{N_i})$ to be the set of the critical points $\delta\in N_i$ of $f_{N_i}$
with Morse index $k$ and $f_{N_i}(\delta)<0$.
We put  $Cr_k^+(f_N):=\bigcup_{i=1}^mCr_k^+(f_{N_i})$ and 
$Cr_k^-(f_N):=\bigcup_{i=1}^mCr_k^-(f_{N_i})$.

Let $X_f$ be the gradient vector field on $M\setminus N$ of a cone end Morse function
$f$ with respect to a cone end Riemannian metric $g$. 
Then the restriction of $X_f$ on the collar neighborhood is 
\[
X_f|_{N_i\times(0,1)}=X_{f_{N_i}}+2f_{N_i}r\frac{\partial}{\partial r}.
\]
Hence we define a vector field $\overline{X}_f$ on $M$ by
\[
\overline{X}_f:=
\left\{
\begin{array}{ll}
X_f, & \mbox{on } M\setminus N, 
\\
X_{f_{N_i}}, & \mbox{on } N_i\times\{0\},
\end{array}
\right.
\]
and denote by 
$\overline{\varphi} _t:M\to M$ the isotopy of $-\overline{X}_f$, i.e. $\overline{\varphi}_t$ is given by
$d\overline{\varphi}_t/dt=-\overline{X}_f\circ \overline{\varphi}_t$
and
$\overline{\varphi}_0(x)=x$.

Let $B^k:=\{(x_1,\ldots,x_k):x_1^2+\cdots+x_k^2<1\} $ be the $k$-dimensional open ball,
and $\partial B^k:=\overline{B^k}\setminus B^k$.
Moreover, we define the $k$-dimensional open half-ball
$H^k:=\{(x_1,\ldots,x_k):x_1^2+\cdots+x_k^2<1, x_k\geq 0\}$
and  $\partial H^k:=\{(x_1,\ldots,x_k)\in H^k:x_k=0\}$.

Now we define stable manifolds and unstable manifolds for critical points.
For $p\in Cr_k(f)$, we define the stable manifold $S_p$ of $p$ by
\[
S_p:=\left\{ x\in M:\lim_{t\to+\infty}\overline{\varphi}_t(x)=p\right\},
\] 
and the unstable manifold $U_p$ of $p$ by 
\[
U_p:=\left\{x\in M:\lim_{t\to-\infty}\overline{\varphi}_t(x)=p\right\}.
\]
Since $\overline{X}_f$ is tangent to $N$, 
$S_p$ and $U_p$ are contained in $M\setminus N$. 
Note that $S_p$ is diffeomorphic to $B^{n-k}$,
and $U_p$ is diffeomorphic to $B^k$.
Moreover, $S_p$ and $U_p$ intersect transversely at $p$.
We fix orientations of $S_p$ and $U_p$ so that 
the orientations of $T_pS_p\oplus T_pU_p$ and $T_pM$ coincide.

Next, for $\gamma\in Cr_k^+(f_{N_i})$, 
we define the stable manifold $S_{\gamma}$ of $\gamma$ by
\[
S_{\gamma}:=\left\{x\in M:\lim_{t\to+\infty}\overline{\varphi}_t(x)=(\gamma,0)
\in N_i\times \{0\}\right\},
\]
and the unstable manifold $U_{\gamma}$ of $\gamma$ by
\[
U_{\gamma}:=\left\{x\in M:\lim_{t\to-\infty}\overline{\varphi}_t(x)=(\gamma,0)
\in N_i\times\{0\}\right\}.
\]
Since $f_{N_i}(\gamma)>0$,
$S_{\gamma}$ is contained in $M$, 
and $U_{\gamma}$ is contained in $N_i\times\{0\}\subset M$.
Note that $U_{\gamma}$ is diffeomorphic to $B^k$,
$S_{\gamma}$ is diffeomorphic to $H^{n-k}$,
and $S_{\gamma}\cap (N_i\times\{0\})$ is diffeomorphic to 
$\partial H^{n-k}\cong B^{n-1-k}$.
Moreover $S_{\gamma}$ and 
$U_{\gamma}$ intersect transversely 
at $(\gamma,0)\in N_i\times\{0\}$.

We fix orientations of $S_{\gamma}$ and $U_{\gamma}$ so that 
the orientations of $T_{\gamma}S_{\gamma}\oplus T_{\gamma}U_{\gamma}$ 
and $T_{\gamma}M$ coincide.
Moreover, we fix an orientation of $S_{\gamma}\cap(N_i\times\{0\})$ so that, 
for an oriented basis $\{v_1,\ldots,v_{n-1-k}\}$ of $T_{\gamma}(S_{\gamma}\cap (N_i\times\{0\}))$
and an outward-pointing vector $v_{out}\in T_{\gamma}M$,
the orientations of $\{v_{out}, v_1,\ldots,v_{n-1-k}\}$ and $T_{\gamma}S_{\gamma}$ coincide.
Then the orientations of $T_{\gamma}(S_{\gamma}\cap (N_i\times\{0\}))
\oplus T_{\gamma}U_{\gamma}$ and $T_{\gamma}N_i$ coincide.

Similarly, for $\delta\in Cr_k^-(f_{N_i})$, 
we define the stable manifold $S_{\delta}$ of $\delta$ by
\[
S_{\delta}:=\left\{x\in M:\lim_{t\to+\infty}\overline{\varphi}_t(x)=(\delta,0)
\in N_i\times\{0\}\right\},
\]
and the unstable manifold $U_{\delta}$ of $\delta$ by
\[
U_{\delta}:=\left\{x\in M:\lim_{t\to-\infty}\overline{\varphi}_t(x)=(\delta,0)
\in N_i\times\{0\}\right\}.
\]
Since $f_{N_i}(\delta)<0$,
$S_{\delta}$ is contained in $N_i\times\{0\}\subset M$, and $U_{\delta}$ is contained in $M$.
Note that $S_{\delta}$ is diffeomorphic to $B^{n-1-k}$,
$U_{\delta}$ is diffeomorphic to $H^{k+1}$,
and $U_{\delta}\cap (N_i\times\{0\})$ is diffeomorphic to $\partial H^{k+1}\cong B^k$.
Moreover, $S_{\delta}$ and $U_{\delta}$ intersect transversely at $(\delta,0)\in N_i\times\{0\}$.

We fix orientations of $S_{\delta}$ and $U_{\delta}$ so that 
the orientations of $T_{\delta}S_{\delta}\oplus T_{\delta}U_{\delta}$ and $T_{\delta}M$ coincide.
Moreover, we fix an orientation of $U_{\delta}\cap(N_i\times\{0\})$ so that, 
for an oriented basis $\{v_{n-k},\ldots,v_n\}$ of $T_{\delta}(U_{\delta}\cap (N_i\times\{0\}))$
and an outward-pointing vector $v_{out}\in T_{\delta}M$,
the orientations of $\{v_{out}, v_{n-k},\ldots,v_n\}$ and $T_{\delta}U_{\delta}$ coincide.
Then the difference of the orientations of $T_{\delta}S_{\delta}
\oplus T_{\delta}(U_{\delta}\cap(N_i\times\{0\}))$ and $T_{\delta}N_i$ is $(-1)^{n-k-1}$.

\section{\bf Handlebody decompositions}

Let $M$ be an $n$-dimensional oriented compact manifold with boundary $N$ as before, and 
$g$ and $f$ a cone end Riemannian metric and 
a cone end Morse function on $M\setminus N$, respectively.
Moreover we assume that $f$ satisfies the Morse--Smale conditions in the following sense: 
\begin{itemize}
\item for $p,p'\in \bigcup_{k=0}^nCr_k(f)$, 
$U_p$ and $S_{p'}$ intersect transversely in $M\setminus N$,
\item for $\theta, \theta'\in \bigcup_{k=0}^{n-1}Cr_k^+(f_{N_i})\cup 
\bigcup_{k=0}^{n-1}Cr_k^-(f_{N_i})$,
$U_{\theta}$ and $S_{\theta'}$ intersect transversely in $N_i$,
\item for $p\in\bigcup_{k=0}^nCr_k(f)$ and $\gamma\in \bigcup_{k=0}^{n-1}Cr_k^+(f_N)$,
$U_p$ and $S_{\gamma}$ intersect transversely in $M\setminus N$,
\item for $\delta\in \bigcup_{k=0}^{n-1}Cr_k^-(f_N)$ and $p\in\bigcup_{k=0}^nCr_k(f)$,
$U_{\delta}$ and $S_p$ intersect transversely in $M\setminus N$, and
\item for $\delta\in \bigcup_{k=0}^{n-1}Cr_k^-(f_{N_i})$ and 
$\gamma\in \bigcup_{k=0}^{n-1}Cr_k^+(f_{N_j})$ with $c_i>c_j$,
$U_{\delta}$ and $S_{\gamma}$ intersect transversely in $M\setminus N$.
\end{itemize}
In fact we can prove that generic cone end Morse functions satisfy the above Morse--Smale conditions
by the standard generosity arguments.

Recall that 
$\overline{X}_f$ is the vector field on $M$ defined by
\[
\overline{X}_f:=
\left\{
\begin{array}{ll}
X_f, & \mbox{on } M\setminus N, 
\\
X_{f_{N_i}}, & \mbox{on } N_i\times\{0\}.
\end{array}
\right.
\]
We call a map $u:\mathbb{R}\to M$ a gradient trajectory from $x$ to $y$
if $du/dt=-\overline{X}_f\circ u$ with $\lim_{t\to-\infty}u(t)=x$ and $\lim_{t\to\infty}u(t)=y$.
Then we can prove the following lemma:

\begin{lem}
Let $f$ be a cone end Morse--Smale function on $M\setminus N$.
For $p\in Cr_k(f)$ and $p'\in Cr_l(f)$, there is no non-constant
gradient trajectory from $p$ to $p'$ if $k\leq l$.
\end{lem}

\begin{proof}
Let $u:\mathbb{R}\to M$ be a non-constant gradient trajectory from $p$ to $p'$. 
Then the image of $u$ is contained in $U_p\cap S_{p'}$.
Since $U_p$ and $S_{p'}$ intersect transversely in $M\setminus N$ 
so that $\dim U_p\cap S_{p'}=k-l$,
and since the dimension of the image of $u$ is 1, 
there is no such gradient trajectory if $k-l\leq 0$.
\end{proof}

Similarly we can prove the following lemma.
We omit the proof:

\begin{lem}
Let $f$ be a cone end Morse--Smale function on $M$.
\\
(1) For $\theta\in Cr_k^+(f_{N_i})\cup Cr_k^-(f_{N_i})$ 
and $\theta'\in Cr_l^+(f_{N_i})\cup Cr_l^-(f_{N_i})$,
there is no non-constant gradient trajectory from $\theta$ to $\theta'$ 
if $k\leq l$.
\\
(2) For $p\in Cr_k(f)$ and $\gamma\in Cr^+_l(f_N)$, 
there is no non-constant gradient trajectory from $p$ to $\gamma$
if $k\leq l$.
\\
(3) For $\delta\in Cr_k^-(f_N)$ and $p\in Cr_l(f)$, 
there is no non-constant gradient trajectory from $\delta$ to $p$
if $k+1\leq l$.
\\
(4) For $\delta\in Cr_k^-(f_{N_i})$ and 
$\gamma\in Cr_l^+(f_{N_j})$ with $i \neq j$,
there is no non-constant gradient trajectory 
from $\delta$ to $\gamma$ if $k+1\leq l$.
\end{lem}

Now we construct a handlebody decomposition of $M$:

\begin{thm}
Let $f$ be a cone end Morse--Smale function on $M\setminus N$.
Then there exists a sequence of open subsets $M^{-1}=\tilde{M}^0=\emptyset
\subset M^0\subset \tilde{M}^1\subset M^1\subset
\cdots \subset\tilde{M}^n\subset  M^n=M$ such that
\begin{itemize}
\item $\partial \tilde{M}^k:=\overline{\tilde{M}^k}\setminus \tilde{M}^k$ 
and $\partial M^k:=\overline{M^k}\setminus M^k$ 
are smooth and transversal to 
$\overline{X}_f$,
where $\overline{\tilde{M}^k}$ and $\overline{M^k}$ are 
the closures of $\tilde{M}^k$ and $M^k$ in $M$, respectively,
\item for $\delta\in Cr_{k-1}^-(f_N)$, $\partial M^{k-1}$ and $U_{\delta}$ intersect transversely,
and $U_{\delta}\setminus M^{k-1}$ is diffeomorphic to the $k$-dimensional closed half-ball,
\item $\overline{M^{k-1}}\cup \bigcup_{\delta\in Cr_{k-1}^-(f_N)}U_{\delta}$
is a deformation retract of $\overline{\tilde{M}^{k}}$;
\item for $p\in Cr_k(f)$,
$\partial \tilde{M}^k$ and $U_p$ intersect transversely, and 
$U_p\setminus \tilde{M}^k$ is diffeomorphic to the $k$-dimensional closed ball,
\item for $\gamma\in Cr_k^+(f_N)$,
$\partial \tilde{M}^k$ and $U_{\gamma}$ intersect transversely, and 
$U_{\gamma}\setminus \tilde{M}^k$ is diffeomorphic to the $k$-dimensional closed ball, and
\item $\overline{\tilde{M}^k}\cup
\bigcup_{p\in Cr_k(f)}U_p\cup\bigcup_{\gamma\in Cr^+_k(f_N)}U_{\gamma}$ 
is a deformation retract of $\overline{M^k}$.
\end{itemize}
\end{thm}

We call the sequence 
 $M^{-1}=\emptyset\subset M^0\subset M^1\subset \cdots \subset M^n=M$ 
a handlebody decomposition of $M$. 

\begin{proof}
We construct $\tilde{M}^k$ and $M^k$ inductively. 
For $p\in Cr_0(f)$,
let $(x_1,\ldots,x_n)$ be a local coordinate centered at $p$ in $M$ and
$B_{\varepsilon}(p):=\{x_1^2+\cdots+x_n^2<\varepsilon^2\}\subset M$;
and for $\gamma\in Cr_0^+(f_N)$,
let $(y_1,\ldots,y_{n-1})$ be a local coordinate centered at $\gamma$ in $N$
and $B_{\varepsilon}(\gamma):=\{y_1^2+\cdots+y_{n-1}^2+r^2<\varepsilon^2\}\subset M$.
Here we put 
$\partial B_{\varepsilon}(\gamma):=\{y_1^2+\cdots+y_{n-1}^2+r^2=\varepsilon^2, r\geq 0\}$.
Then we may take $\varepsilon$ to be small 
so that the closures of $B_{\varepsilon}(p)$ and $B_{\varepsilon}(\gamma)$ in $M$ 
are mutually disjoint,
and $\partial B_{\varepsilon}(p)$ and $\partial B_{\varepsilon}(\gamma)$
are transversal to $\overline{X}_f$.
We define 
$M^0:=\bigcup_{p\in Cr_0(f)}B_{\varepsilon}(p)\cup 
\bigcup_{\gamma\in Cr_0^+(f_N)}B_{\varepsilon}(\gamma)$, and then
\begin{itemize}
\item $\bigcup_{p\in Cr_0(f)}U_p\cup \bigcup_{\gamma\in Cr_0^+(f_N)}U_{\gamma}$
is a deformation retract of $\overline{M^0}$, and 
\item $\partial M^0$ 
is smooth and transversal to $\overline{X}_f$.
\end{itemize}
This is the first step of $k=0$ to construct the handlebody decomposition of $M$.
Suppose we have $M^{-1}=\tilde{M}^0=\emptyset \subset M^0\subset 
\tilde{M}^1\subset M^1\subset \cdots \subset \tilde{M}^{k-1}\subset M^{k-1}$
as in the theorem.
Since $\partial M^{k-1}$ is smooth and transversal to $\overline{X}_f$, for $\delta\in Cr_{k-1}^-(f_N)$, 
\begin{itemize}
\item $\partial M^{k-1}$ and $U_{\delta}$ intersect transversely,
\end{itemize}
and moreover, since $f$ is Morse--Smale,
\begin{itemize}
\item $U_{\delta}\setminus M^{k-1}$
is diffeomorphic to the $k$-dimensional closed half-ball,
\end{itemize}
where the $k$-dimensional closed half-ball is diffeomorphic to 
$\{y_1^2+\cdots+y_{k-1}^2+r^2\leq 1,r\geq 0\}$.
Hence we may attach half $k$-handles for $\delta\in Cr_{k-1}^-(f_N)$ to $\overline{M^{k-1}}$
and obtain $\tilde{M}^k$ so that
\begin{itemize}
\item $\overline{M^{k-1}}\cup \bigcup_{\delta\in Cr_{k-1}^-(f_N)}U_{\delta}$
is a deformation retract of $\overline{\tilde{M}^{k}}$, and
\item $\partial \tilde{M}^k$ is smooth and transversal to $\overline{X}_f$,
\end{itemize}
where the half $k$-handle is diffeomorphic to 
$\{y_1^2+\cdots+y_{k-1}^2+r^2\leq 1, r\geq 0\}\times B^{n-k}$
and the attaching map is from $\{y_1^2+\cdots +y_{k-1}^2+r^2=1, r\geq0\}\times B^{n-k}$ 
to $\partial M^{k-1}$.
Since $\partial \tilde{M}^k$ is transversal to $\overline{X}_f$,
\begin{itemize}
\item for $p\in Cr_k(f)$,
$\partial \tilde{M}^k$ and $U_p$ intersect transversely, and 
\item for $\gamma\in Cr_k^+(f_N)$,
$\partial \tilde{M}^k$ and $U_{\gamma}$ intersect transversely,
\end{itemize}
and moreover, since $f$ is Morse--Smale, 
\begin{itemize}
\item $U_p\setminus \tilde{M}^k$ is diffeomorphic to the $k$-dimensional closed ball, and
\item $U_{\gamma}\setminus \tilde{M}^k$ is diffeomorphic to the $k$-dimensional closed ball.
\end{itemize}
Hence we may attach $k$-handles for $p\in Cr_k(f)$ and $\gamma\in Cr_k^+(f_N)$ to 
$\overline{\tilde{M}^k}$ and obtain $M^k$ so that
\begin{itemize}
\item $\overline{\tilde{M}^k}\cup \bigcup_{p\in Cr_k(f)}U_p\cup 
\bigcup_{\gamma\in Cr_k^+(f_N)}U_{\gamma}$ is a deformation retract of $\overline{M^k}$, and
\item $\partial M^k$ is smooth and transversal to $\overline{X}_f$.
\end{itemize}
Then these $\tilde{M}^k$ and $M^k$ satisfy the conditions as in the theorem.
Therefore we obtain the handlebody decomposition 
$M^{-1}=\emptyset\subset M^0\subset M^1\subset \cdots \subset M^n=M$
by induction.
\end{proof}

\section{\bf Relative cycles}

In this section we introduce relative cycles of critical points 
to define our Morse complex on manifolds with boundary.
This new technique is also very important for the future applications.

First we confirm notation.
Let $M$ be an $n$-dimensional oriented compact manifold with boundary $N$ as before, and 
$g$ and $f$ a cone end Riemannian metric and 
a cone end Morse--Smale function on $M\setminus N$, respectively.
Moreover, let $M^{-1}=\tilde{M}^0=\emptyset \subset M^0\subset \tilde{M}^1\subset M^1
\subset \cdots \subset \tilde{M}^n\subset M^n=M$ be the sequence of open subsets constructed in 
Theorem 4.3.

Let $B^k_r:=\{(x_1,\ldots,x_k):x_1^2+\cdots+x_k^2<r^2\}$ 
and $\partial B^k_r:=\overline{B^k_r}\setminus B^k_r$,
and similarly, let
$H^k_r:=\{(x_1,\ldots,x_k):x_1^2+\cdots+x_k^2<r^2, x_k\geq 0\}$
and $\partial H^k_r:=\{(x_1,\ldots,x_k)\in H^k_r: x_k=0\}$.

For $0<\varepsilon<1$, we define a diffeomorphism 
$\rho_{\varepsilon}:M\to \rho_{\varepsilon}(M)\subset M$ so that
$\rho_{\varepsilon}(x)=x$ for $x\notin[0,\varepsilon)\times N$,
and $\rho_{\varepsilon}(x,r)=(x,r/2+\varepsilon/2)$ for $(x,r)\in N\times [0,\varepsilon/2)$.

For $p\in Cr_k(f)$, we fix a diffeomorphism $\psi_p:B^k_1\to U_p$ with $\psi_p(0)=p$.
Since $\partial \tilde{M}^k$ and $U_p$ intersect transversely, and 
$U_p\setminus \tilde{M}^k$ is diffeomorphic to the $k$-dimensional closed ball,
there exists $0<r_p<1$ such that $\psi_p(\partial B^k_{r_p})\subset \tilde{M}^k$.
We call the restriction 
$\psi_p|_{\overline{B^k_{r_p}}}: \overline{B^k_{r_p}}\to U_p$ 
a relative cycle for $p$, and denote by
$\sigma_p:\overline{B^k_{r_p}}\to U_p$.
Note that $\sigma_p$ is an embedding,
and $S_p$ and the image of $\sigma_p$ intersect transversely 
and positively at $p$.
Similarly, for $\gamma\in Cr_k^+(f_N)$, 
we fix a diffeomorphism $\psi_{\gamma}:B^k_1\to U_{\gamma}$ with $\psi_{\gamma}(0)=\gamma$.
Since 
\begin{itemize}
\item $\overline{\tilde{M}^k}\cup
\bigcup_{p\in Cr_k(f)}U_p\cup \bigcup_{\gamma\in Cr_k^+(f_N)}U_{\gamma}$
is a deformation retract of $\overline{M^k}$, and 
\item $U_{\gamma}\setminus \tilde{M}^k$ is diffeomorphic to the $k$-dimensional closed ball,
\end{itemize}
there exist $0<\varepsilon<1$ and $0<r_{\gamma}<1$ such that 
$\rho_{\varepsilon}\circ \psi_{\gamma}(B^k_{r_{\gamma}})\subset M^k$
and 
$\rho_{\varepsilon}\circ\psi_{\gamma}(\partial B^k_{r_{\gamma}})\subset \tilde{M}^k$.
Note that, since $\partial \tilde{M}^k$ is smooth and transversal to $\overline{X}_f$,
$\varphi_t(\rho_{\varepsilon}\circ \psi_{\gamma}(\partial B^k_{r_{\gamma}}))$ is contained in
$\tilde{M}^k$, for $t\geq 0$,
where $\varphi_t$ is the isotopy of $-X_f$.
Moreover, since $\varepsilon>0$ and
\begin{itemize}
\item $\overline{M^{k-1}}\cup\bigcup_{\delta\in Cr_{k-1}^-(f_N)}U_{\delta}$
is a deformation retract of $\overline{\tilde{M}^k}$, and 
\item $\partial M^{k-1}$ are smooth and transversal to $\overline{X}_f$,
\end{itemize}
there exists $T_{\gamma}>0$ such that 
$\varphi_{T_{\gamma}}(\rho_{\varepsilon}\circ \psi_{\gamma}(\partial B^k_{r_{\gamma}}))
\subset M^{k-1}$.
Now we define a map $\sigma_{\gamma}:
\overline{B^k_{r_{\gamma}}}\cup (\partial B^k_{r_{\gamma}}\times [0,T_{\gamma}])\to M$ as follows:
first we glue $\overline{B^k_{r_{\gamma}}}$ and $\partial B^k_{r_{\gamma}}\times [0,T_{\gamma}]$
by the natural identification of 
$\partial B^k_{r_{\gamma}}
\subset \overline{B^k_{r_{\gamma}}}$ 
with $\partial B^k_{r_{\gamma}}\times \{0\}
\subset\partial B^k_{r_{\gamma}}\times [0,T_{\gamma}]$;
and then we define $\sigma_{\gamma}$ by 
$\sigma_{\gamma}(x):=\rho_{\varepsilon}\circ\psi_{\gamma}(x)$
if $x\in \overline{B^k_{r_{\gamma}}}$,
and $\sigma_{\gamma}(x,t):=\varphi_t(\rho_{\varepsilon}\circ\psi_{\gamma}(x))$ 
if $(x,t)\in\partial B^k_{r_{\gamma}}\times [0,T_{\gamma}]$.
We call $\sigma_{\gamma}:\overline{B^k_{r_{\gamma}}}
\cup (\partial B^k_{r_{\gamma}}\times[0,T_{\gamma}])
\to M$ a relative cycle of $\gamma$.
Note that $\sigma_{\gamma}$ is a piecewise embedding,
and $S_{\gamma}$ and the image of $\sigma_{\gamma}$ intersect transversely
and positively at $\rho_{\varepsilon}(\gamma)$. 

\section{\bf Gradient trajectories}

Let $M$ be an $n$-dimensional oriented compact manifold with boundary $N$ as before, and 
$g$ and $f$ a cone end Riemannian metric and 
a cone end Morse--Smale function on $M\setminus N$, respectively.

Let $p,p'\in M\setminus N$ be critical points of $f$, and
$u:\mathbb{R}\to M\setminus N$ a map which satisfies $du/dt=-X_f\circ u$
with $\lim_{t\to-\infty}u(t)=p$ and $\lim_{t\to\infty}u(t)=p'$.
We call such $u$ a gradient trajectory from $p$ to $p'$,
and moreover, we call such $u$ up to parameter shift an unparameterized gradient trajectory.
We define $\mathcal{M}(p,p')$ 
to be the set of the unparameterized gradient trajectories from $p$ to $p'$.
Since an intersection point $x\in S_{p'}\cap \sigma_p(\partial B^k_{r_p})$ corresponds to 
the unparameterized gradient trajectory from $p$ to $p'$ through $x$,
$\mathcal{M}(p,p')$ can be identified with $S_{p'}\cap \sigma_p(\partial B^k_{r_p})$.
Let $p\in Cr_k(f)$ and $p'\in Cr_{k-1}(f)$.
For $x\in S_{p'}\cap \sigma_p(\partial B^k_{r_p})$, we define $\epsilon_x:=1$ if the orientations of
$T_xS_{p'}\oplus T_x\sigma_p(\partial B^k_{r_p})$ and $T_xM$ coincide, 
and $\epsilon_x:=-1$ otherwise.
Then we assign $\epsilon_x$ to $u\in\mathcal{M}(p,p')$ 
passing through $x\in S_{p'}\cap \sigma_p(\partial B^k_{r_p})$,
and put $\sharp\mathcal{M}(p,p'):=\sum_{x\in S_{p'}\cap \sigma_p(\partial B^k_{r_p})}\epsilon_x$,
which is nothing but the intersection number of $S_{p'}$ and $\sigma_p(\partial B^k_{r_p})$.

Similarly, for $p\in Cr_k(f)$ and $\gamma\in Cr_{k-1}^+(f_N)$, we define $\mathcal{M}(p,\gamma)$
to be the set of the unparameterized gradient trajectories $u:\mathbb{R}\to M\setminus N$
which satisfies $du/dt=-X_f\circ u$
with $\lim_{t\to-\infty}u(t)=p$ and $\lim_{t\to\infty}u(t)=(\gamma,0)\in N\times\{0\}\subset M$;
and we define $\sharp\mathcal{M}(p,\gamma)$ to be the intersection number of
$S_{\gamma}$ and $\sigma_p(\partial B^k_{r_p})$;
and moreover, for $\delta\in Cr_{k-1}^-(f_N)$ and $p\in Cr_{k-1}(f)$,
we define $\mathcal{M}(p,\gamma)$
to be the set of the unparameterized gradient trajectories $u:\mathbb{R}\to M\setminus N$
which satisfies $du/dt=-X_f\circ u$
with $\lim_{t\to-\infty}u(t)=(\delta,0)\in N\times \{0\}\subset M$ and $\lim_{t\to\infty}u(t)=p$;
and 
for $\delta\in Cr_{k-1}^-(f_{N_i})$ and $\gamma\in Cr_{k-1}^+(f_{N_j})$ with $c_i>c_j$,
we define $\mathcal{M}(\delta,\gamma)$
to be the set of the unparameterized gradient trajectories $u:\mathbb{R}\to M\setminus N$
which satisfies $du/dt=-X_f\circ u$
with $\lim_{t\to-\infty}u(t)=(\delta,0)\in N_i\times \{0\}\subset M$ and 
$\lim_{t\to\infty}u(t)=(\gamma,0)\in N_j\times \{0\}\subset M$;
and we define $\sharp\mathcal{M}(\delta,p)$ and $\sharp\mathcal{M}(\delta,\gamma)$
to be the intersection numbers, similarly.

On the other hand, for $\gamma\in Cr_k^+(f_{N_i})$ and 
$\gamma'\in Cr_{k-1}^+(f_{N_i})$, 
we define $\mathcal{N}(\gamma,\gamma')$
to be the set of the unparameterized gradient trajectories $u:\mathbb{R}\to N_i$
which satisfies $du/dt=-X_{f_{N_i}}\circ u$
with $\lim_{t\to-\infty}u(t)=\gamma$ and $\lim_{t\to\infty}u(t)=\gamma'$.
We fix a diffeomorphism $\psi_{\gamma}:B^k_1\to U_{\gamma}$ with $\psi_{\gamma}(0)=\gamma$.
Let $0<r<1$.
For $x\in S_{\gamma'}\cap \psi_{\gamma}(\partial B^k_r)$,
we define $\epsilon_x:=1$ if the orientations of 
$T_xS_{\gamma'}\oplus T_x\psi_{\gamma}(\partial B^k_r)$
and $T_xM$ coincide, and $\epsilon_x:=-1$ otherwise.
Then we assign $\epsilon_x$ to $u\in\mathcal{N}(\gamma,\gamma')$ 
passing through $x\in S_{\gamma'}\cap \psi_{\gamma}(\partial B^k_r)$,
and put $\sharp\mathcal{N}(\gamma,\gamma'):=
\sum_{x\in S_{\gamma'}\cap \psi_{\gamma}(\partial B^k_r)}\epsilon_x$,
which is nothing but the intersection number of $S_{\gamma'}$ and 
$\psi_{\gamma}(\partial B^k_r)$.
Similarly, for $\delta\in Cr_k^-(f_{N_i})$ and 
$\delta'\in Cr_{k-1}^-(f_{N_i})$, 
we define $\mathcal{N}(\delta,\delta')$
to be the set of the unparameterized gradient trajectories $u:\mathbb{R}\to N_i$
which satisfies $du/dt=-X_{f_{N_i}}\circ u$
with $\lim_{t\to-\infty}u(t)=\delta$ and $\lim_{t\to\infty}u(t)=\delta'$;
and we define $\sharp\mathcal{N}(\delta,\delta')$ to be the intersection number of
$S_{\delta'}$ and $\psi_{\delta}(\partial B^k_r)$.

For $\gamma\in Cr_k^+(f_{N_i})$ and 
$\delta\in Cr_{k-1}^-(f_{N_i})$, 
we define $\mathcal{N}(\gamma,\delta)$
to be the set of the unparameterized gradient trajectories $u:\mathbb{R}\to N_i$
which satisfies $du/dt=-X_{f_{N_i}}\circ u$
with $\lim_{t\to-\infty}u(t)=\gamma$ and $\lim_{t\to\infty}u(t)=\delta$.
For $x\in S_{\delta}\cap \psi_{\gamma}(\partial B^k_r)\subset N_i$,
let $v_{in}\in T_xM$ be an inward-pointing vector,
$\{v_1,\ldots,v_{n-k+1}\}$ an oriented basis of $T_xS_{\delta}$,
and $\{v_{n-k+2},\ldots,v_n\}$ an oriented basins of $T_x\psi_{\gamma}(\partial B^k_r)$.
We define $\epsilon_x:=1$ if the orientations of 
$\{v_1,\ldots,v_{n-k},v_{in},v_{n-k+2},\ldots,v_n\}$
and $T_xM$ coincide, and $\epsilon_x:=-1$ otherwise.
Then we assign $\epsilon_x$ to $u\in\mathcal{N}(\gamma,\delta)$ 
passing through $x\in S_{\delta}\cap \psi_{\gamma}(\partial B^k_r)$,
and put $\sharp\mathcal{N}(\gamma,\delta):=
\sum_{x\in S_{\delta}\cap \psi_{\gamma}(\partial B^k_r)}\epsilon_x$.

We remark that, for $\delta\in Cr_k^-(f_{N_i})$ and $\gamma\in Cr_{k-1}^+(f_{N_i})$,
there is no map $u:\mathbb{R}\to N_i$
which satisfies $du/dt=-X_{f_{N_i}}\circ u$
with $\lim_{t\to-\infty}u(t)=\delta$ and $\lim_{t\to\infty}u(t)=\gamma$
since $f_{N_i}(\delta)<0<f_{N_i}(\gamma)$.

\section{\bf Morse homology of manifolds with boundary}

Finally we recall our Morse homology of manifolds with boundary, introduced in \cite{ak1},
and re-certify that the Morse homology is isomorphic to the absolute singular homology
through connecting homomorphisms.

Let $M$ be an $n$-dimensional oriented compact manifold with boundary $N$ as before, and 
$g$ and $f$ a cone end Riemannian metric and 
a cone end Morse--Smale function on $M\setminus N$, respectively. 
Moreover, let $M^{-1}=\emptyset \subset M^0\subset M^1\subset \cdots \subset M^n=M$ 
be the handlebody decomposition of $M$ as in Theorem 4.3.
Then, owing to the conditions of $M^k$, 
\begin{equation}
H_l(M^k,M^{k-1};\mathbb{Z})
=\begin{cases}
\bigoplus_{p\in Cr_k(f)}\mathbb{Z}[\sigma_p]\oplus 
\bigoplus_{\gamma\in Cr_k^+(f_N)}\mathbb{Z}[\sigma_{\gamma}], & l=k,
\\
0, & {\rm otherwise},
\end{cases}
\end{equation}
where $[\sigma_p]$ and $[\sigma_{\gamma}]$ are the relative cycles 
$\sigma_p:(\overline{B^k_{r_p}},\partial B^k_{r_p})\to (M^k,M^{k-1})$ and
$\sigma_{\gamma}:(\overline{B_{r_{\gamma}}^k}
\cup(\partial B_{r_{\gamma}}^k\times[0,T_{\gamma}]),
\partial B_{r_{\gamma}}^k\times \{T_{\gamma}\})\to (M^k,M^{k-1})$.

We denote by $\delta_k:H_k(M^k,M^{k-1};\mathbb{Z})\to H_{k-1}(M^{k-1},M^{k-2};\mathbb{Z})$
the connecting homomorphism.
The connecting homomorphisms satisfy $\delta_{k-1}\circ\delta_k=0$, 
and we obtain a chain complex 
$(H_*(M^*,M^{*-1};\mathbb{Z}),\delta_*)$.
Because of (4) we can prove that the homology of $(H_*(M^*,M^{*-1};\mathbb{Z}),\delta_*)$
is isomorphic to the absolute singular homology of $M$ 
in the same way as CW decompositions.

On the other hand,
for a relative cycle $[\sigma:(\Sigma,\partial \Sigma)\to (M^k,M^{k-1})]$, 
the connecting homomorphism $\delta_k$ can be written as
\[
\delta_k[\sigma:(\Sigma,\partial \Sigma)\to (M^k,M^{k-1})]
=[\sigma:(\partial \Sigma,\emptyset)\to (M^{k-1},M^{k-2})].
\]
Since $S_p$ and the image of $\sigma_p$ intersect transversely and positively at $p$, 
we may think of $S_p$ as the dual base of $[\sigma_p]$,
and similarly since $S_{\gamma}$ and the image of $\sigma_{\gamma}$ intersect transversely
and positively at $\rho_{\varepsilon}(\gamma)$,
we may think of $S_{\gamma}$ as the dual base of $[\sigma_{\gamma}]$.  
Hence $\delta_k$ can be written as
\begin{eqnarray*}
\lefteqn{\delta_k[\sigma:(\Sigma,\partial \Sigma)\to(M^k,M^{k-1})]}
\\
&=&
\sum_{p\in Cr_{k-1}(f)}\sharp(S_p\cap \sigma(\partial \Sigma))[\sigma_p]
+\sum_{\gamma\in Cr_{k-1}^+(f_N)}
\sharp(S_{\gamma}\cap \sigma(\partial \Sigma))[\sigma_{\gamma}],
\end{eqnarray*}
where $\sharp(S_p\cap \sigma(\partial \Sigma))$ and 
$\sharp(S_{\gamma}\cap \sigma(\partial \Sigma))$
are the intersection numbers of $S_p$ and $\sigma(\partial \Sigma)$, and 
$S_{\gamma}$ and $\sigma(\partial \Sigma)$, respectively.
This description was essentially given by J. Milnor in \cite{mi}.

Now we define a free $\mathbb{Z}$ module $CM_k(f)$ by
\[
CM_k(f):=\bigoplus_{p\in Cr_k(f)}\mathbb{Z}p\oplus 
\bigoplus_{\gamma\in Cr_k^+(f_N)}\mathbb{Z}\gamma,
\]
which is isomorphic to $H_k(M^k,M^{k-1};\mathbb{Z})$ by identifying $p$ and $\gamma$ with 
the relative cycles $[\sigma_p]$ and $[\sigma_{\gamma}]$, respectively,
and define a linear map $\partial_k:CM_k(f)\to CM_{k-1}(f)$ by, 
for $p\in Cr_k(f)$,
\begin{eqnarray*}
\partial_k p&:=&
\sum_{p'\in Cr_{k-1}(f)}\sharp\mathcal{M}(p,p')p'
+\sum_{\gamma\in Cr_{k-1}^+(f_N)}\sharp\mathcal{M}(p,\gamma)\gamma,
\end{eqnarray*}
and for $\gamma\in Cr_k^+(f_N)$,
\begin{eqnarray*}
\partial_k\gamma &:=&
\sum_{p\in Cr_{k-1}(f)}\sum_{\delta\in Cr_{k-1}^-(f_N)}
\sharp\mathcal{N}(\gamma,\delta)\sharp\mathcal{M}(\delta,p)p
\\
&&
+\sum_{\gamma'\in Cr_{k-1}^+(f_N)}\sum_{\delta\in Cr_{k-1}^-(f_N)}
\sharp\mathcal{N}(\gamma,\delta)\sharp\mathcal{M}(\delta,\gamma')\gamma'
\\
&&
+\sum_{\gamma'\in Cr_{k-1}^+(f_N)}\sharp\mathcal{N}(\gamma,\gamma')\gamma'.
\end{eqnarray*}

Again our main theorem is that:

\begin{thm}
$(CM_*(f),\partial_*)$ is a chain complex, i.e. $\partial_{k-1}\circ\partial_k=0$, 
and the homology is isomorphic to the {\it absolute} singular homology of~$M$.
\end{thm}

\begin{proof}
We already know $CM_k(f)\cong H_k(M^k,M^{k-1};\mathbb{Z})$, and hence 
we show that $\partial_k=\delta_k$.
For $p\in Cr_k(f)$, 
by the descriptions of the connecting homomorphisms and the moduli spaces of gradient trajectories,
\begin{eqnarray*}
\lefteqn{\delta_k[\sigma_p]}
\\
&=&\sum_{p'\in Cr_{k-1}(f)}\sharp(S_{p'}\cap \sigma_p(\partial R^k_{r_p}))
[\sigma_{p'}]
+
\sum_{\gamma\in Cr_{k-1}^+(f_N)}\sharp(S_{\gamma}\cap \sigma_p(\partial R^k_{r_p}))
[\sigma_{\gamma}]
\\
&=&
\sum_{p'\in Cr_{k-1}(f)}\sharp\mathcal{M}(p,p')[\sigma_{p'}]
+
\sum_{\gamma\in Cr_{k-1}^+(f_N)}\sharp\mathcal{M}(p,\gamma)[\sigma_{\gamma}].
\end{eqnarray*}
Hence $\partial_kp=\delta_k[\sigma_p]$ under the identification of $p$ and $\gamma$
with $[\sigma_p]$ and $[\sigma_{\gamma}]$, respectively.

Recall that the relative cycle 
$\sigma_{\gamma}:\overline{B^k_{r_{\gamma}}}\cup(\partial B^k_{r_{\gamma}}\times [0,T_{\gamma}])
\to M$ for $\gamma\in Cr_k^+(f_N)$ is given by
$\sigma_{\gamma}(x):=\rho_{\varepsilon}\circ\psi_{\gamma}(x)$ 
if $x\in \overline{B^k_{r_{\gamma}}}$,
and $\sigma_{\gamma}(x,t):=\varphi_t(\rho_{\varepsilon}\circ\psi_{\gamma}(x))$ 
if $(x,t)\in\partial B^k_{r_{\gamma}}\times [0,T_{\gamma}]$.
Then, for $\gamma\in Cr_k^+(f_N)$, $\delta_k[\sigma_{\gamma}]$ can be written as
\begin{eqnarray*}
\delta_k[\sigma_{\gamma}]
&=&\sum_{p\in Cr_{k-1}(f)}\sharp(S_p\cap 
\sigma_{\gamma}(\partial B^k_{r_{\gamma}}\times\{T_{\gamma}\})
[\sigma_p]
\\
&&
+\sum_{\gamma'\in Cr_{k-1}^+(f_N)}\sharp(S_{\gamma'}\cap 
\sigma_{\gamma}(\partial B^k_{r_{\gamma}}\times\{T_{\gamma}\})
[\sigma_{\gamma'}].
\end{eqnarray*}

Let $\gamma\in Cr_k^+(f_{N_i})$.
For $p\in Cr_{k-1}(f)$,
an intersection point $x\in S_p\cap 
\sigma_{\gamma}(\partial B^k_{r_{\gamma}}\times\{T_{\gamma}\})$ corresponds to a pair of
\begin{itemize}
\item a gradient trajectory $u:(-\infty, r_{\gamma}]\to N_i$
which satisfies $du/dt=-X_{f_{N_i}}\circ u$ 
with $\lim_{t\to -\infty}u(t)=\gamma$, and
\item a gradient trajectory $v:[0, \infty)\to M\setminus N$
passing through $x$
which satisfies $dv/dt=-X_f\circ v$
with $v(0)=\rho_{\varepsilon}(u(r_{\gamma}))$ and $\lim_{t\to \infty}v(t)=p$.
\end{itemize}
As $\varepsilon\to 0$, $v$ breaks into two pieces: 
\begin{itemize}
\item one is a gradient trajectory $u':[0,\infty)\to N_i$
which satisfies $du'/dt=-X_{f_{N_i}}\circ u'$ 
with $u'(0)=u(r_{\gamma})$ and 
$\lim_{t\to -\infty}u(t)=\delta\in Cr_{k-1}^-(f_{N_i})$, and
\item the other is an unparameterized  gradient trajectory $v':\mathbb{R}\to M\setminus N$
which satisfies $dv'/dt=-X_f\circ v'$
with $\lim_{t\to-\infty}v'(t)=(\delta,0)\in N_i\times\{0\}$ and $\lim_{t\to \infty}v'(t)=p$.
\end{itemize}
Note that $u$ and $u'$ give an unparameterized gradient trajectory $u'':\mathbb{R}\to N_i$
which satisfies $du''/dt=-X_{f_{N_i}}\circ u''$
with $\lim_{t\to-\infty}u''(t)=\gamma$ and $\lim_{t\to \infty}u''(t)
\\
=\delta$.
Then an intersection point $x\in S_p\cap 
\sigma_{\gamma}(\partial B^k_{r_{\gamma}}\times\{T_{\gamma}\})$
gives $(u'',v')\in \mathcal{N}(\gamma,\delta)\times \mathcal{M}(\delta,p)$.
Conversely, by the gluing analysis, 
$(u'',v')\in \mathcal{N}(\gamma,\delta)\times \mathcal{M}(\delta,p)$ gives 
$x\in S_p\cap \sigma_{\gamma}(\partial B^k_{r_{\gamma}}\times\{T_{\gamma}\})$.
Hence 
\[
\sharp(S_p\cap 
\sigma_{\gamma}(\partial B^k_{r_{\gamma}}\times\{T_{\gamma}\})
=\sharp\mathcal{N}(\gamma,\delta)\sharp\mathcal{M}(\delta,p).
\]

Let $\gamma\in Cr_k^+(f_{N_i})$.
If $\gamma'\in Cr_{k-1}^+(f_{N_j})$ with $c_i>c_j$,
an intersection point $x\in S_{\gamma'}\cap 
\sigma_{\gamma}(\partial B^k_{r_{\gamma}}\times\{T_{\gamma}\})$ corresponds to a pair of
\begin{itemize}
\item a gradient trajectory $u:(-\infty, r_{\gamma}]\to N_i$
which satisfies $du/dt=-X_{f_{N_i}}\circ u$ 
with $\lim_{t\to -\infty}u(t)=\gamma$, and
\item a gradient trajectory $v:[0, \infty)\to M\setminus N$
passing through $x$
which satisfies $dv/dt=-X_f\circ v$
with $v(0)=\rho_{\varepsilon}(u(r_{\gamma}))$ and $\lim_{t\to \infty}v(t)=(\gamma',0)\in N_j\times\{0\}$.
\end{itemize}
As $\varepsilon\to 0$, $v$ breaks into two pieces: 
\begin{itemize}
\item one is a gradient trajectory $u':[0,\infty)\to N_i$
which satisfies $du'/dt=-X_{f_{N_i}}\circ u'$ 
with $u'(0)=u(r_{\gamma})$ and 
$\lim_{t\to \infty}u(t)=\delta\in Cr_{k-1}^-(f_{N_i})$, and
\item the other is an unparameterized  gradient trajectory $v':\mathbb{R}\to M\setminus N$
which satisfies $dv'/dt=-X_f\circ v'$
with $\lim_{t\to-\infty}v'(t)=(\delta,0)\in N_i\times\{0\}$ and
 $\lim_{t\to \infty}v'(t)=(\gamma',0)\in N_j\times\{0\}$.
\end{itemize}
Note that $u$ and $u'$ give an unparameterized gradient trajectory $u'':\mathbb{R}\to N_i$
which satisfies $du''/dt=-X_{f_{N_i}}\circ u''$
with $\lim_{t\to-\infty}u''(t)=\gamma$ and $\lim_{t\to \infty}u''(t)
\\
=\delta$.
Then an intersection point $x\in S_{\gamma'}\cap 
\sigma_{\gamma}(\partial B^k_{r_{\gamma}}\times\{T_{\gamma}\})$
gives $(u'',v')\in \mathcal{N}(\gamma,\delta)\times \mathcal{M}(\delta,\gamma')$.
Conversely, by the gluing analysis, 
$(u'',v')\in \mathcal{N}(\gamma,\delta)\times \mathcal{M}(\delta,\gamma')$ gives 
$x\in S_{\gamma'}\cap \sigma_{\gamma}(\partial B^k_{r_{\gamma}}\times\{T_{\gamma}\})$.
Hence 
\[
\sharp(S_{\gamma'}\cap 
\sigma_{\gamma}(\partial B^k_{r_{\gamma}}\times\{T_{\gamma}\})
=\sharp\mathcal{N}(\gamma,\delta)\sharp\mathcal{M}(\delta,\gamma').
\]
If $\gamma'\in Cr_{k-1}^+(f_{N_i})$, 
an intersection point $x\in S_{\gamma'}\cap 
\sigma_{\gamma}(\partial B^k_{r_{\gamma}}\times\{T_{\gamma}\})$ corresponds to a pair of
\begin{itemize}
\item a gradient trajectory $u:(-\infty, r_{\gamma}]\to N_i$
which satisfies $du/dt=-X_{f_{N_i}}\circ u$ 
with $\lim_{t\to -\infty}u(t)=\gamma$, and
\item a gradient trajectory $v:[0, \infty)\to M\setminus N$
passing through $x$
which satisfies $dv/dt=-X_f\circ v$
with $v(0)=\rho_{\varepsilon}(u(r_{\gamma}))$ and $\lim_{t\to \infty}v(t)=(\gamma',0)\in N_i\times\{0\}$.
\end{itemize}
This time, as $\varepsilon\to 0$, $v$ converges to 
a gradient trajectory $v':[0,\infty)\to N_i$
which satisfies $dv'/dt=-X_{f_{N_i}}\circ v'$ 
with $v'(0)=u(r_{\gamma})$ and 
$\lim_{t\to \infty}v'(t)=\gamma'$
because, if $v$ converged to a pair of the following two maps:
\begin{itemize}
\item a gradient trajectory $u':[0,\infty)\to N_i$
which satisfies $du'/dt=-X_{f_{N_i}}\circ u'$ 
with $u'(0)=u(r_{\gamma})$ and 
$\lim_{t\to \infty}u(t)=\delta\in Cr_{k-1}^-(f_{N_i})$, and
\item an unparameterized  gradient trajectory $v'':\mathbb{R}\to M\setminus N$
which satisfies $dv''/dt=-X_f\circ v''$
with $\lim_{t\to-\infty}v''(t)=(\delta,0)\in N_i\times\{0\}$ and
 $\lim_{t\to \infty}v''(t)=(\gamma',0)\in N_i\times\{0\}$,
\end{itemize}
the existence of such $v''$ contradicts Lemma 2.5.
Note that $u$ and $v'$ give an unparameterized gradient trajectory $u'':\mathbb{R}\to N_i$
which satisfies $du''/dt=-X_{f_{N_i}}\circ u''$
with $\lim_{t\to-\infty}u''(t)=\gamma$ and $\lim_{t\to \infty}u''(t)=\gamma'$.
Then an intersection point $x\in S_{\gamma'}\cap 
\sigma_{\gamma}(\partial B^k_{r_{\gamma}}\times\{T_{\gamma}\})$
gives $u''\in \mathcal{N}(\gamma,\gamma')$;
and conversely,  
$u''\in \mathcal{N}(\gamma,\gamma')$ gives 
$x\in S_{\gamma'}\cap \sigma_{\gamma}(\partial B^k_{r_{\gamma}}\times\{T_{\gamma}\})$.
Hence 
\[
\sharp(S_{\gamma'}\cap 
\sigma_{\gamma}(\partial B^k_{r_{\gamma}}\times\{T_{\gamma}\})
=\sharp\mathcal{N}(\gamma,\gamma').
\]

Therefore, we obtain
\begin{eqnarray*}
\delta_k[\sigma_{\gamma}] &=&
\sum_{p\in Cr_{k-1}(f)}\sum_{\delta\in Cr_{k-1}^-(f_N)}
\sharp\mathcal{N}(\gamma,\delta)\sharp\mathcal{M}(\delta,p)[\sigma_p]
\\
&&
+\sum_{\gamma'\in Cr_{k-1}^+(f_N)}\sum_{\delta\in Cr_{k-1}^-(f_N)}
\sharp\mathcal{N}(\gamma,\delta)\sharp\mathcal{M}(\delta,\gamma')[\sigma_{\gamma'}]
\\
&&
+\sum_{\gamma'\in Cr_{k-1}^+(f_N)}\sharp\mathcal{N}(\gamma,\gamma')[\sigma_{\gamma'}],
\end{eqnarray*}
which implies that  $\partial_k\gamma
=\delta_k[\sigma_{\gamma}]$ under the identification of $p$ and $\gamma$
with $[\sigma_p]$ and $[\sigma_{\gamma}]$, respectively.
\end{proof}

\end{document}